\documentclass[11pt]{article}
\usepackage{amssymb}
\usepackage{amsthm}
\usepackage{amsmath}
\usepackage[dvips]{epsfig}
\usepackage{graphicx}
\usepackage{color}
\usepackage{url}

\newtheorem{theorem}{Theorem}
\newtheorem{lemma}{Lemma}

\newtheorem{conjecture}{Conjecture}
\newtheorem{proposition}{Proposition}
\newtheorem{remark}{Remark}

\newcommand{\R}{\mathbb{R}}
\newcommand{\C}{\mathbb{C}}

\newcommand{\N}{\mathbb{N}}

\begin{document}

\title{\bf Orbital stability of Dirac solitons}

\author{Dmitry E. Pelinovsky and Yusuke Shimabukuro \\
{\small \it Department of Mathematics and Statistics, McMaster
University, Hamilton, Ontario, Canada, L8S 4K1 } }

\date{\today}
\maketitle

\begin{abstract}
We prove $H^1$ orbital stability of Dirac solitons in the integrable massive
Thirring model by working with an additional conserved quantity which complements
Hamiltonian, momentum and charge functionals of the general nonlinear Dirac equations.
We also derive a global bound on the $H^1$ norm of the $L^2$-small solutions of the
massive Thirring model.
\end{abstract}


\section{Introduction}

Nonlinear Dirac equations are considered as a relativistic version of the nonlinear
Schr\"{o}dinger (NLS) equation. Compared to the NLS equation, proofs of global existence and orbital stability of
solitary waves are complicated by the fact that the quadratic part of the Hamiltonian of the nonlinear Dirac equations is
not bounded from neither above nor below. Similar situation occurs in a gap between two bands of continuous
spectrum in the Schr\"{o}dinger equations with a periodic potential, for which the nonlinear Dirac equations
are justified rigorously as an asymptotic  model (Chapter 2.2 in \cite{Pel-book}).

Because orbital stability of solitary waves is not achieved by the standard energy arguments \cite{GSS},
researchers have studied spectral and asymptotic stability of solitary waves in many details.
Spectral properties of linearized Dirac operators were studied by a combination of analytical
methods and numerical approximations \cite{GD,ChPel-cme,Comech,COM1,Comech-alone1,Comech-alone2,COM2}.
Asymptotic stability of small solitary waves in the general nonlinear Dirac equations was
studied with dispersive estimates both in the space of one \cite{PS,Komech,Kopylova} and three \cite{Boussaid,Boussaid2,BC12}
dimensions. Global existence and scattering to zero for small initial data were obtained
again in one \cite{HN1,HN2} and three \cite{Machihara,Machihara2} dimensions.

When nonlinear Dirac equations are considered in one spatial dimension, a particular
attention is drawn to the massive Thirring model (MTM) \cite{Thirring}, which is known
to be integrable with the inverse scattering transform method \cite{KN,KM}.
In laboratory coordinates, this model takes the following form:
\begin{equation}
\label{MTM}
\left\{ \begin{array}{cc}
i (u_t + u_x) + v = 2 |v|^2 u, \\
i (v_t - v_x) + u = 2 |u|^2 v,
\end{array} \right.
\end{equation}
where $(u,v)(x,t) : \R \times \R_+ \to \C^2$.

Selberg and Tesfahun \cite{ST} proved local well-posedness of the MTM system in $H^s(\R)$ for $s > 0$
and global well-posedness in $H^s(\R)$ for $s > \frac{1}{2}$.
Machihara {\em et al.} \cite{Machihara3} proved for similar nonlinear Dirac equations with
quadratic nonlinear terms that local well-posedness holds in $H^s(\R)$ for $s > -\frac{1}{2}$
and that the Cauchy problem is ill-posed in $H^{-1/2}(\R)$.
Candy \cite{Candy} proved local and global
well-posedness of the MTM system in $L^2(\R)$. These results do not rely on
the inverse scattering transform for the MTM system.

On the other hand, by using the inverse scattering transform, spectral stability
of the MTM solitons was established by Kaup \& Lakoba \cite{KL1}. The stationary MTM solitons
are known in the exact analytical form:
\begin{equation}
\label{MTM-solitons}
\left\{ \begin{array}{c} u = U_{\omega}(x+x_0) e^{i \omega t + i\alpha}, \\
v = \bar{U}_{\omega}(x+x_0) e^{i \omega t + i \alpha}, \end{array} \right.
\end{equation}
with
\begin{equation}
\label{MTM-form}
U_{\omega}(x) = \frac{\sqrt{1 - \omega^2}}{\sqrt{1+\omega} \cosh\left(\sqrt{1 - \omega^2} x\right) +
i \sqrt{1-\omega} \sinh\left(\sqrt{1 - \omega^2} x\right)},
\end{equation}
where $\alpha$ and $x_0$ are real parameters related to the gauge and space translations, whereas
$\omega \in (-1,1)$ is a parameter that determines where the MTM solitons are placed in the gap
between two branches of the continuous spectrum of the linearized MTM system (\ref{MTM}).
Perturbations of the MTM system and the loss of spectral stability of solitary waves was consequently considered using
the spectral representations \cite{BarPel,KL2} and the Evans function \cite{KapSand}.
No results on the orbital or asymptotic stability of the MTM solitons (\ref{MTM-solitons})
in the time evolution of the MTM system (\ref{MTM}) have been obtained so far.

The idea for our work relies on the existence of an infinite set of conserved quantities
in the MTM system, which has been known for quite some time \cite{KM}. The three
standard conserved quantities for the nonlinear Dirac equations
are related to the translational invariance of the system with respect to gauge, space, and time
transformations. For the MTM system (\ref{MTM}), these three conserved quantities are
referred to as the charge $Q$, momentum $P$, and Hamiltonian $H$ functionals:
\begin{equation}
\label{quantity-Q} Q = \int_{\mathbb{R}} \left( |u|^2+|v|^2
\right) dx,
\end{equation}
\begin{equation}
\label{quantity-P} P = \frac{i}{2} \int_{\mathbb{R}} \left(
u \bar{u}_x - u_x \bar{u} + v \bar{v}_x - v_x \bar{v} \right) dx,
\end{equation}
and
\begin{equation}
\label{quantity-H} H = \frac{i}{2} \int_{\mathbb{R}} \left(
u \bar{u}_x - u_x \bar{u} - v \bar{v}_x + v_x \bar{v} \right) dx +
\int_{\R} \left( -v \bar{u} - u \bar{v} + 2|u|^2 |v|^2 \right) dx.
\end{equation}
The charge $Q$ is useful to establish global bound on the $L^2$ norm of solutions,
as soon as the local existence in $L^2(\R)$ is proven \cite{Candy}. The other two functionals
$P$ and $H$ are defined in $H^{1/2}(\R)$ but they are not so useful because of the fact that
the quadratic part of the Hamiltonian $H$ is not bounded from neither above nor below.
Therefore, although the MTM solitons (\ref{MTM-solitons}) are critical points of the functional
$H + \omega Q + c P$, where $\omega \in (-1,1)$ is a parameter of the MTM solitons (\ref{MTM-form})
and $c = 0$ for stationary solutions, this functional cannot serve as a Lyapunov functional for orbital stability
or instability of the MTM solitons.

Nevertheless, we find another conserved quantity of the MTM system (\ref{MTM}),
thanks to the integrability via the inverse scattering transform method:
\begin{eqnarray}
\nonumber
R & = & \int_{\mathbb{R}} \left[ |u_{x}|^2+|v_x|^2
- \frac{i}{2}(u_x\overline{u}-\overline{u}_xu)(|u|^2+2|v|^2) +
\frac{i}{2}(v_x\overline{v}-\overline{v}_xv)(2|u|^2+|v|^2) \right. \\
\label{quantity-R}
& \phantom{t} & \left. -(u\overline{v}+\overline{u}v)(|u|^2+|v|^2)+2|u|^2|v|^2(|u|^2+|v|^2)\right]dx.
\end{eqnarray}
Derivation of the conserved quantity $R$ is reviewed in Appendix A.
The conserved quantity $R$ is well defined in $H^1(\R)$ and we shall use it
to prove orbital stability of the MTM solitons in $H^1(\R)$. The main result of this article is
the following theorem.

\begin{theorem}
There is $\omega_0 \in (0,1]$ such that
for any fixed $\omega \in (-\omega_0,\omega_0)$, the MTM soliton $(u,v) = (U_{\omega},\bar{U}_{\omega})$
is a local non-degenerate minimizer of $R$ in $H^1(\R,\C^2)$ under the constraints of fixed values of $Q$ and $P$.
Therefore, the MTM soliton is orbitally stable in $H^1(\R,\C^2)$ with respect to the time evolution of the MTM system (\ref{MTM}).
\label{theorem-main}
\end{theorem}

From a technical point, we establish that the MTM solitons (\ref{MTM-solitons}) are critical points
of the functional $\Lambda_{\omega} = R + (1-\omega^2) Q$, where $\omega \in (-1,1)$ is
the same parameter of the MTM solitons (\ref{MTM-form}). By using operator calculus in constrained spaces,
we prove that there is $\omega_0 \in (0,1]$ such that
for any fixed $\omega \in (-\omega_0,\omega_0)$, the functional
$\Lambda_{\omega}$ is strictly convex at the MTM soliton
$(u,v) = (U_{\omega},\bar{U}_{\omega})$
under the constraints of fixed values of $Q$ and $P$. As a result, $\Lambda_{\omega}$ can
serve as a Lyapunov functional for orbital stability of the MTM solitons
thanks to the conservation of $R$, $Q$, and $P$ and the standard analysis
of orbital stability of solitary waves \cite{GSS}. Appendix B states relevant results used
in our work.

Note that the non-degenerate minimizer in Theorem \ref{theorem-main} can be translated
along two ``trivial" parameters $\alpha$ and $x_0$ in (\ref{MTM-solitons}). These parameters
are related to the gauge and space translations and can be excluded by additional constraints
on the perturbations to the MTM soliton $(u,v) = (U_{\omega},\bar{U}_{\omega})$. Whereas
we have not succeeded to find the exact value of $\omega_0$, we conjecture that $\omega_0 = 1$,
that is, the result of Theorem \ref{theorem-main} extends to the entire family of MTM solitons.

We also mention some recent relevant results.
First, orbital stability of breathers
of the modified KdV equation is proved in space $H^2(\mathbb{R})$ in the recent work of Alejo and Munoz \cite{AM}. 
An additional conserved quantity is introduced and used to complement conservation of the 
Hamiltonian and momentum of the modified KdV equation.
This work is conceptually similar to the ideas of our paper with the following difference. 
It uses the known
characterization of orbital stability of  multi-solitons in the KdV and modified KdV equations
with higher-order conserved quantities \cite{Perelman,MS}, whereas our work introduces 
a new concept of orbital stability of Dirac solitons.

Second, a different technique involving additional conserved quantities is proposed in 
the work of Deconinck and Kapitula \cite{DecKap}, where no constraints are imposed to study orbital stability of
periodic waves with respect to perturbations of multiple period in the KdV equation. The lack of minimizing properties
of the higher-order Hamiltonian is corrected by adding lower-order Hamiltonians with specially selected strength parameter.

As a bi-product of our work, we obtain a global apriori bound on
$L^2$-small solutions of the MTM system in $H^1(\R)$. The standard apriori bounds
on the $H^1$ norm of the solution of a general nonlinear Dirac equation
grow at a double-exponential rate \cite{GWH,Pel-survey}.
At the present time, we do not know if global bounds on the $H^1$ norm can be proven for
all (not $L^2$-small) solutions of the MTM system (\ref{MTM}).

We also do not know if scattering to the MTM solitons (and hence asymptotic stability of the MTM solitons)
can be proved by using the inverse scattering transform methods, e.g. the
auto--B\"{a}cklund transformation, similar to what was done recently for NLS solitons
\cite{DP,MizPel}. (Scattering to zero solutions
and criterion for the absence of the MTM solitons were established earlier in
\cite{HN1,HN2} and \cite{Pel-survey}, respectively.) These problems remain open for further studies.

The paper is organized as follows. Global bounds on the $H^1$ norm of the $L^2$-small solutions are obtained
in Section 2. In Section 3, we prove that the MTM solitons are non-degenerate
minimizers of the functional $\Lambda_{\omega}$ in $H^1$ under the constraints
of fixed $Q$ and $P$ for $\omega \in (-\omega_0,\omega_0)$
with some $\omega_0 \in (0,1]$. Appendix A reports derivation of the conserved quantity $R$.
Appendix B lists a number of technical results without proofs.

\section{Global $H^1$ bound on the $L^2$-small solutions}

Thanks to local existence in $L^2$ \cite{Candy} and the conservation of $Q$, the $L^2$ solutions are extended
globally for all $t \in \R$ with a global bound on the $L^2$ norm of the solution.
On the other hand, local existence holds also in $H^n(\R)$ for any integer $n \in \N$ but
the apriori bounds on the $H^n$ norm grows at a super-exponential rate \cite{GWH,Pel-survey}. Here we
use the conservation of $R$ to obtain the global bound on the $H^1$ norm of the solution,
which holds for all $L^2$-small solutions. The following theorem gives the main result
of this section.

\begin{theorem}
There is a $Q_0 > 0$ such that for all $(u_0,v_0) \in H^1(\R)$ with
$\| u_0 \|_{L^2}^2 + \| v_0 \|_{L^2}^2 \leq Q_0$, there is a positive
$(u_0,v_0)$-dependent constant $C(u_0,v_0)$ such that
\begin{equation}
\label{bound-global}
\| u \|_{H^1}^2 + \| v \|_{H^1}^2 \leq C(u_0,v_0),
\end{equation}
for all $t \in \R$.
\label{theorem-bound}
\end{theorem}

\begin{proof}
By Sobolev embedding of $H^1(\R)$ into $L^p(\R)$ for any $p \geq 2$, the values of
$Q$ and $R$ are finite if $(u_0,v_0) \in H^1(\R)$. To obtain the assertion of the theorem,
we need to show that the value of $R$ gives an upper bound for the value of
$\| \partial_x u \|_{L^2}^2 + \| \partial_x v \|_{L^2}^2$.
Then, the standard approximation argument in Sobolev space $H^2(\R)$ yields a conservation
of $R$ from the balance equation [see equation (\ref{balance-R}) in Appendix A],
whereas the standard continuation argument
yields the global bound (\ref{bound-global}).

The lower bound for $R$ follows from two applications of the Gagliardo--Nirenberg inequality
in one dimension. For any $p \geq 1$, there is a constant $C_p > 0$ such that
\begin{equation}
\label{GN-inequality}
\| u \|^{2p}_{L^{2p}} \leq C_p \| \partial_x u \|_{L^2}^{p-1} \| u \|_{L^2}^{1+p},
\quad u \in H^1(\R).
\end{equation}

First, we note that quadratic and sixth-order terms are positive definite.
To control the last fourth-order terms of $R$ from below, we note that
there is a positive constant $C$ (which may change from line to another line) such that
\begin{eqnarray*}
\left| \int_{\R} (|u|^2 + |v|^2) (v \bar{u} + u \bar{v}) dx \right| & \leq & C \left( \| u \|^4_{L^4} +
\| v \|_{L^4}^4 \right) \\
& \leq & C \left( \| u \|_{L^2}^3 + \| v \|^3_{L^2} \right)
\left( \| \partial_x u \|_{L^2} + \| \partial_x v \|_{L^2} \right).
\end{eqnarray*}
Because the positive part of $R$ is quadratic in $\| \partial_x u \|_{L^2}$ and $\| \partial_x v \|_{L^2}$,
the previous bound is sufficient to control the last fourth-order terms of $R$ from below.
On the other hand, the first fourth-order terms like
\begin{eqnarray*}
\left| \int_{\R} (u_x\overline{u}-\overline{u}_xu) |u|^2 dx \right| \leq C \| u \|_{L^6}^3 \| \partial_x u \|_{L^2}
\leq C \| u \|_{L^2}^2 \| \partial_x u \|^2_{L^2}
\end{eqnarray*}
can only be controlled from below if $\| u \|_{L^2}^2$ is sufficiently small. This proves the assertion of the theorem.
\end{proof}

\begin{remark}
One can try to squeeze the first fourth-order terms of $R$ between the positive quadratic and sixth-order terms of $R$. For example,
if reduction $v = \bar{u}$ is used, these terms of $R$ are estimated from below by
\begin{eqnarray*}
2 \| \partial_x u \|^2_{L^2} + 4 \| u \|_{L^6}^6 - 3 i \int_{\R} (u_x\overline{u}-\overline{u}_xu) |u|^2 dx
\geq 2 \| \partial_x u \|^2_{L^2} + 4 \| u \|_{L^6}^6 - 6  \| \partial_x u \|_{L^2} \| u \|_{L^6}^3.
\end{eqnarray*}
Unfortunately, the lower bound is not positive definite. Therefore, we do not know if the global bound (\ref{bound-global})
can be extended to all (not necessarily $L^2$-small) solutions of the MTM system (\ref{MTM}).
\end{remark}

\section{$H^1$ orbital stability of solitons}

Critical points of the energy functional $H + \omega Q$ with fixed $\omega \in (-1,1)$
satisfy the system of first-order differential equations
\begin{equation}
\label{MTM-ODE}
\left\{ \begin{array}{cc}
+ i \frac{d u}{d x} - \omega u + v = 2 |v|^2 u, \\
- i \frac{d v}{d x} - \omega v + u = 2 |u|^2 v.
\end{array} \right.
\end{equation}
The stationary MTM solitons (\ref{MTM-solitons}) correspond to
the reduction $u = U_{\omega}$ and $v = \bar{U}_{\omega}$,
where $U_{\omega}$ is a solution of the first-order differential equation
\begin{equation}
\label{ODE}
i \frac{d U}{d x} - \omega U + \bar{U} = 2 |U|^2 U.
\end{equation}
Integrating this differential equation with the zero boundary conditions,
we obtain MTM solitons in the explicit form (\ref{MTM-form}).

\begin{remark}
Two translational parameters of the MTM solitons (\ref{MTM-solitons})
are obtained from the gauge and space translations:
\begin{equation}
\label{translations}
u(x,t) \mapsto u(x+x_0,t) e^{i \alpha}, \quad v(x,t) \mapsto v(x+x_0,t) e^{i \alpha},
\end{equation}
where $\alpha$ and $x_0$ are real-valued. On the other hand, more general moving MTM solitons
must have another parameter of velocity $c \in (-1,1)$, which can be recovered by
using the Lorentz transformation:
\begin{eqnarray}
\label{Lorentz}
\left\{ \begin{array}{l} u(x,t) \mapsto \left( \frac{1 + c}{1 - c} \right)^{1/4}
u\left(\frac{x-ct}{\sqrt{1-c^2}},\frac{t - cx}{\sqrt{1-c^2}}\right), \\
v(x,t) \mapsto  \left( \frac{1 - c}{1 + c} \right)^{1/4}
v\left(\frac{x-ct}{\sqrt{1-c^2}},\frac{t - cx}{\sqrt{1-c^2}}\right).
\end{array} \right.
\end{eqnarray}
In what follows, without the loss of generality,
we simplify our consideration by working with the stationary MTM solitons for $c = 0$.
\end{remark}

Critical points of the energy functional $R + \Omega Q$ satisfy the system of
second-order differential equations
\begin{equation}
\label{MTM-ODE-second-order}
\left\{ \begin{array}{cc}
\frac{d^2 u}{d x^2} + 2 i (|u|^2 + |v|^2) \frac{d u}{dx} + 2 i u v \frac{d \bar{v}}{dx}
- 2 |v|^2 (2 |u|^2 + |v|^2) u + (2 |u|^2 + |v|^2) v + u^2 \bar{v} = \Omega u, \\
\frac{d^2 v}{d x^2} - 2 i (|u|^2 + |v|^2) \frac{d v}{dx} - 2 i u v \frac{d \bar{u}}{dx}
- 2 |u|^2 (|u|^2 + 2 |v|^2) v + (|u|^2 + 2 |v|^2) u + v^2 \bar{u} = \Omega v.
\end{array} \right.
\end{equation}
Using the reduction $u = U$ and $v = \bar{U}$,
we obtain a second-order differential equation
\begin{equation}
\label{ODE-second-order}
\frac{d^2 U}{d x^2} + 6 i |U|^2 \frac{d U}{dx}
- 6 |U|^4 U + 3 |U|^2 \bar{U} + U^3 = \Omega U.
\end{equation}

Substituting the first-order equation (\ref{ODE}) to the second-order equation (\ref{ODE-second-order}) yields
the constraint
$$
(1-\omega^2) U + \left(2 |U|^4 + 2 \omega |U|^2 - U^2 - \bar{U}^2 \right) U = \Omega U,
$$
which is satisfied by the MTM soliton $U = U_{\omega}$ in the explicit form (\ref{MTM-form})
if $\Omega = 1 - \omega^2$. Therefore, the MTM soliton (\ref{MTM-form}) is a critical point of the
energy functional
\begin{equation}
\label{Lyapunov}
\Lambda_{\omega} := R + (1-\omega^2) Q, \quad \omega \in (-1,1)
\end{equation}
in the energy space $H^1(\R,\C^2)$.

We shall now prove that there is $\omega_0 \in (0,1]$ such that
for any fixed $\omega \in (-\omega_0,\omega_0)$,
the critical point of $\Lambda_{\omega}$
is a local non-degenerate minimizer in the constrained space $X_{\omega}$,
which is defined as an orthogonal complement
in $H^1(\R,\C^2)$ of the following complex-valued constraints:
\begin{eqnarray}
\label{constraint-1}
(u,v) \in \C^2 : \quad \int_{\R} \left( \bar{U}_{\omega} u + U_{\omega} v \right) dx & = & 0, \\
\label{constraint-2}
(u,v) \in \C^2 : \quad \int_{\R} \left( \bar{U}'_{\omega} u + U'_{\omega} v \right) dx & = & 0,
\end{eqnarray}
where the prime denotes the derivative of $U_{\omega}$ with respect to $x$.

The real part of the constraint (\ref{constraint-1}) is equivalent to the condition
that the conserved quantity $Q$ is fixed under the perturbation $(u,v)$ to the MTM
soliton $(U_{\omega},\bar{U}_{\omega})$ at the first order. The imaginary part
of the constraint (\ref{constraint-1}) represents the orthogonality of
the perturbation $(u,v,\bar{u},\bar{v})$ to the following eigenvector of
a linearization operator for the zero eigenvalue,
\begin{equation}
\label{eigenvector-1}
{\bf F}_g := \left[ \begin{array}{c} i U_{\omega} \\ i \bar{U}_{\omega} \\
- i \bar{U}_{\omega} \\ -i U_{\omega} \end{array} \right],
\end{equation}
which is induced by the gauge translation of the MTM soliton $(U_{\omega},\bar{U}_{\omega},\bar{U}_{\omega},U_{\omega})$
related to the parameter $\alpha$ in the transformation (\ref{translations}).

Similarly, the imaginary part of the constraint (\ref{constraint-2}) is equivalent to the condition
that the conserved quantity $P$ is fixed under the perturbation $(u,v)$ to the MTM
soliton $(U_{\omega},\bar{U}_{\omega})$ at the first order. The real part
of the constraint (\ref{constraint-2}) represents the orthogonality of
the perturbation $(u,v,\bar{u},\bar{v})$ to the following eigenvector of
a linearization operator for the zero eigenvalue,
\begin{equation}
\label{eigenvector-2}
{\bf F}_s := \left[ \begin{array}{c} U_{\omega}' \\ \bar{U}_{\omega}' \\
\bar{U}_{\omega}' \\ U_{\omega}' \end{array} \right],
\end{equation}
which is induced by the space translation of the MTM soliton $(U_{\omega},\bar{U}_{\omega},\bar{U}_{\omega},U_{\omega})$
related to the parameter $x_0$ in the transformation (\ref{translations}).

The following theorem gives the main result of this section.

\begin{theorem}
There is $\omega_0 \in (0,1]$ such that
for any fixed $\omega \in (-\omega_0,\omega_0)$, the Lyapunov functional
$\Lambda_{\omega}$ defined by (\ref{Lyapunov}) is strictly convex at
$(u,v) = (U_{\omega},\bar{U}_{\omega})$ in the orthogonal complement of the complex-valued
constraints (\ref{constraint-1}) and (\ref{constraint-2}) in $H^1(\R,\C^2)$.
\label{theorem-minimum}
\end{theorem}

To prove Theorem \ref{theorem-minimum}, we use a perturbation $(u,v,\bar{u},\bar{v})$
to the MTM soliton $(U_{\omega},\bar{U}_{\omega},\bar{U}_{\omega},U_{\omega})$
and expand the Lyapunov functional $\Lambda_{\omega}$ to the quadratic form in $(u,v,\bar{u},\bar{v})$,
which is defined by the Hessian operator
\begin{equation}
\label{Hessian}
L = \left[ \begin{array}{cccc} L_1 & 2 L_2 & L_2 & L_3 \\
2 \bar{L}_2 & \bar{L}_1 & \bar{L}_3 & \bar{L}_2 \\
\bar{L}_2 & \bar{L}_3 & \bar{L}_1 & 2 \bar{L}_2 \\
L_3 & L_2 & 2 L_2 & L_1 \end{array} \right],
\end{equation}
where
\begin{eqnarray*}
L_1 & = & -\frac{d^2}{dx^2} - 4 i |U_{\omega}|^2 \frac{d}{dx} - 4 i \bar{U}_{\omega} \frac{d U_{\omega}}{dx} + 10 |U_{\omega}|^4 - 2 U_{\omega}^2
- 2 \bar{U}_{\omega}^2 + 1 - \omega^2, \\
L_2 & = &  -2 i U_{\omega} \frac{d U_{\omega}}{dx} + 4 U_{\omega}^2 |U_{\omega}|^2 - 2 |U_{\omega}|^2, \\
L_3 & = & - 2 i |U_{\omega}|^2 \frac{d}{dx} - 2 i \bar{U}_{\omega} \frac{d U_{\omega}}{dx} + 8 |U_{\omega}|^4 - U_{\omega}^2 - \bar{U}_{\omega}^2.
\end{eqnarray*}

Similarly to the case of linearized Dirac equations \cite{ChPel-cme},
the $4\times 4$ matrix operator $L$ is diagonalized by
two $2 \times 2$ matrix operators $L_{\pm}$ by means of the orthogonal
similarity transformation
\begin{eqnarray*} \label{matrixS}
S^T L S = \left[ \begin{array}{cc} L_+ & 0 \\ 0 & L_- \end{array} \right], \quad \mbox{\rm where} \quad
S =\frac{1}{\sqrt{2}} \left[ \begin{array}{cccc} 1 & 0 & -1 & 0 \\ 0 & 1 & 0 & 1 \\ 0 & 1 & 0 & -1 \\ 1 & 0 & 1 & 0 \end{array} \right].
\end{eqnarray*}
The matrix operators $L_{\pm}$ are found from the reduction of $L$ under the constraints $v = \pm \bar{u}$:
\begin{equation}
\label{Hessian-plus}
L_+ = \left[ \begin{array}{cc} \ell_+ & - 6 \omega U_{\omega}^2 \\ - 6 \omega \bar{U}_{\omega}^2 & \bar{\ell}_+ \end{array} \right], \quad
L_- = \left[ \begin{array}{cc} \ell_- & 2 \omega U_{\omega}^2 \\ 2 \omega \bar{U}_{\omega}^2 & \bar{\ell}_- \end{array} \right],
\end{equation}
where
\begin{eqnarray*}
\ell_+ & = & -\frac{d^2}{dx^2} - 6 i |U_{\omega}|^2 \frac{d}{dx}
+ 6 |U_{\omega}|^4 - 3 U_{\omega}^2 + 3 \bar{U}_{\omega}^2 - 6 \omega |U_{\omega}|^2
+ 1 - \omega^2, \\
\ell_- & = & -\frac{d^2}{dx^2} - 2 i |U_{\omega}|^2 \frac{d}{dx}
- 2 |U_{\omega}|^4 - U_{\omega}^2 + \bar{U}_{\omega}^2 - 2 \omega |U_{\omega}|^2
+ 1 - \omega^2,
\end{eqnarray*}
and the first-order differential equation (\ref{ODE}) for $U_{\omega}$ has been used.
Thanks to the exponential decay of $U_{\omega}$ to $0$ at infinity, by Weyl's Lemma,
the continuous spectrum of operators $L_{\pm}$ is located on the semi-infinite interval $[1-\omega^2,\infty)$
with $1 - \omega^2 > 0$. The following results characterize the discrete spectrum of operators $L_{\pm}$.

\begin{proposition}
For any $\omega \in (-1,1)$, we have
\begin{equation}
\label{subspace-1}
L_+ \left[ \begin{array}{c} U_{\omega}' \\ \bar{U}_{\omega}' \end{array} \right] = \left[ \begin{array}{c} 0 \\ 0 \end{array} \right], \quad
L_- \left[ \begin{array}{c} U_{\omega} \\ -\bar{U}_{\omega} \end{array} \right] = \left[ \begin{array}{c} 0 \\ 0 \end{array} \right],
\end{equation}
which represent the eigenvectors (\ref{eigenvector-1}) and (\ref{eigenvector-2}). In addition, for $\omega = 0$,
there is a zero eigenvalue associated with the eigenvectors
\begin{equation}
\label{subspace-2}
\omega = 0 : \quad L_+ \left[ \begin{array}{c} U_0' \\ -\bar{U}_0' \end{array} \right] =
\left[ \begin{array}{c} 0 \\ 0 \end{array} \right], \quad
L_- \left[ \begin{array}{c} U_0 \\ \bar{U}_0 \end{array} \right] =
\left[ \begin{array}{c} 0 \\ 0 \end{array} \right].
\end{equation}
\end{proposition}

\begin{proof}
Validity of (\ref{subspace-1}) for any $\omega \in (-1,1)$
is obtained by direct substitution with the help of the differential equations (\ref{ODE}) and (\ref{ODE-second-order}).
Eigenvectors (\ref{eigenvector-1}) and (\ref{eigenvector-2}) are related to
the physical symmetries of the MTM system (\ref{MTM}) with respect
to the gauge and space translations.

Because operators $L_{\pm}$ are diagonal for $\omega = 0$, the existence of the eigenvectors (\ref{subspace-2})
follows from the existence of the eigenvectors (\ref{subspace-1}) for $\omega = 0$.
\end{proof}

\begin{lemma}
For any $\omega \in (-1,1)$, operator $L_-$ has exactly two eigenvalues below the continuous spectrum. Besides the zero eigenvalue
associated with the eigenvectors (\ref{subspace-1}), $L_-$ also has a positive eigenvalue for
$\omega \in (0,1)$ and a negative eigenvalue for $\omega \in (-1,0)$, which is
associated with the eigenvector in (\ref{subspace-2}) for $\omega = 0$.
\label{lemma-eigenvalues-1}
\end{lemma}

\begin{proof}
Let us consider the eigenvalue
problem $L_- {\bf u} = \mu {\bf u}$, where ${\bf u} = (u,\bar{u})$
is an eigenvector and $\mu$ is the spectral parameter. Using the transformation
$$
u(x) = \varphi(x) e^{-i \int_0^x |U_{\omega}(x')|^2 dx'}
$$
where $\varphi$ is a new eigenfunction, we obtain an equivalent spectral problem:
$$
\left[ \begin{array}{cc} -\partial_x^2 + 1-\omega^2 -2 \omega |U_{\omega}|^2 - 3 |U_{\omega}|^4 & 2 \omega |U_{\omega}|^2 \\
2 \omega |U_{\omega}|^2 & -\partial_x^2 + 1-\omega^2 -2 \omega |U_{\omega}|^2 - 3 |U_{\omega}|^4 \end{array} \right]
\left[ \begin{array}{c} \varphi \\ \bar{\varphi} \end{array} \right] = \mu \left[ \begin{array}{c} \varphi \\ \bar{\varphi} \end{array} \right],
$$
thanks to the fact that
$$
U_{\omega}^2 e^{2 i \int_0^x |U_{\omega}(x')|^2 dx'} = \frac{1 - \omega^2}{\omega + \cosh(2 \sqrt{1-\omega^2} x)} = |U_{\omega}|^2.
$$

Because the off-diagonal entries are real, we set
$$
\psi_{\pm} := \varphi(x) \pm \bar{\varphi}(x), \quad
z := \sqrt{1-\omega^2} x, \quad
\mu := (1-\omega^2) \lambda
$$
to diagonalize the spectral problem into two uncoupled spectral problems
associated with the linear Schr\"{o}dinger operators:
\begin{equation}
\label{first}
- \frac{d^2 \psi_+}{d z^2} + \left[ 1 - \frac{3 (1-\omega^2)}{(\omega + \cosh(2z))^2} \right] \psi_+ = \lambda \psi_+
\end{equation}
and
\begin{equation}
\label{second}
- \frac{d^2 \psi_-}{d z^2} + \left[ 1 - \frac{3 (1-\omega^2)}{(\omega + \cosh(2z))^2} - \frac{4 \omega}{\omega + \cosh(2z)}
\right] \psi_- = \lambda \psi_-.
\end{equation}
The eigenvector (\ref{subspace-1}) in the kernel of $L_-$ yields the eigenfunction
$$
\psi_0(z) = \frac{1}{(\omega + \cosh(2z))^{1/2}}
$$
of the spectral problem (\ref{second}) for $\lambda = 0$.
Because the eigenfunction $\psi_0$ is positive definite, the simple zero eigenvalue of
the spectral problem (\ref{second}) is at the bottom of the Schr\"{o}dinger
spectral problem for any $\omega \in (-1,1)$, by Sturm's Nodal Theorem (Theorem A in Appendix B).
Furthermore, the function
$$
\psi_c(z) = \frac{\sinh(2z)}{\omega + \cosh(2z)}
$$
corresponds to the end-point resonance at $\lambda = 1$ for the spectral problem
\begin{equation}
\label{second-compared}
- \frac{d^2 \psi}{d z^2} + \left[ 1 - \frac{8 (1-\omega^2)}{(\omega + \cosh(2z))^2} - \frac{4 \omega}{\omega + \cosh(2z)}
\right] \psi = \lambda \psi.
\end{equation}
Because the function $\psi_c$ has exactly one zero, there is only one isolated eigenvalue
below the continuous spectrum for the spectral problem (\ref{second-compared}) (Theorem A in Appendix B).
Now the difference between the potentials of the spectral problems (\ref{second}) and (\ref{second-compared}) is
$$
\Delta V(z) =  \frac{5 (1-\omega^2)}{(\omega + \cosh(2z))^2},
$$
where $\Delta V > 0$  for all $z \in \R$ and $\omega \in (-1,1)$. By Sturm's Comparison Theorem
(Theorem B in Appendix B), a solution of the spectral problem (\ref{second}) for $\lambda = 1$,
which is bounded as $z \to -\infty$, has exactly one zero. Therefore,
the spectral problem (\ref{second}) has exactly one isolated eigenvalue
for all $\omega \in (-1,1)$ and this is the zero eigenvalue with the
eigenfunction $\psi_0$.

The difference between the potentials of the spectral problems (\ref{first}) and (\ref{second}) is
given by
$$
\Delta V(z) = \frac{4 \omega}{\omega + \cosh(2z)}.
$$
Since $\Delta V > 0$ for $\omega \in (0,1)$, the spectral problem (\ref{first})
has precisely one isolated eigenvalue for $\omega \in (0,1)$ (Theorem B in Appendix B)
and this eigenvalue is positive (Theorem C in Appendix B).
On the other hand, since $\Delta V < 0$ for $\omega \in (-1,0)$
and $\psi_0 > 0$ is an eigenfunction of the spectral problem (\ref{second})
for $\lambda = 0$, the spectral problem (\ref{first}) has at least one negative
eigenvalue for $\omega \in (-1,0)$ (Theorem C in Appendix B). To show that this negative eigenvalue
is the only isolated eigenvalue of the spectral problem (\ref{first}),
we note that
$$
\omega + \cosh(2z) \geq \omega + 1 + 2 z^2, \quad z \in \R
$$
and consider the spectral problem
\begin{equation}
\label{first-compared}
- \frac{d^2 \psi}{d z^2} + \left[ 1 - \frac{3 (1-\omega^2)}{(\omega + 1 + 2 z^2)^2} \right] \psi = \lambda \psi.
\end{equation}
Rescaling the independent variable $z := \frac{\sqrt{1+\omega}}{\sqrt{2}} y$ and denoting $\psi(z) := \tilde{\psi}(y)$,
we rewrite (\ref{first-compared}) in the equivalent form
\begin{equation}
\label{first-compared-new}
- \frac{d^2 \tilde{\psi}}{d y^2} - \frac{3}{(1 + y^2)^2} \left(1 - \frac{1 + \omega}{2} \right) \tilde{\psi} =
\frac{(\lambda - 1) (1 + \omega)}{2} \tilde{\psi}.
\end{equation}
It follows that the function
$$
\tilde{\psi}_c(y) = \frac{y}{\sqrt{1 + y^2}}
$$
corresponds to the end-point resonance at $\lambda = 1$ for the spectral problem
\begin{equation}
\label{first-compared-new-yet}
- \frac{d^2 \tilde{\psi}}{d y^2} - \frac{3}{(1 + y^2)^2} \tilde{\psi} =
\frac{(\lambda - 1) (1 + \omega)}{2} \tilde{\psi}.
\end{equation}
Because the function $\tilde{\psi}_c$ has exactly one zero, there is only one isolated eigenvalue
below the continuous spectrum for the spectral problem (\ref{first-compared-new-yet}).
Because the difference between potentials of the spectral problems (\ref{first-compared-new}) and (\ref{first-compared-new-yet})
as well as those of the spectral problems (\ref{first}) and (\ref{first-compared}) is strictly positive
for all $\omega \in (-1,1)$, by Theorem B in Appendix B,
the spectral problem (\ref{first}) has exactly one isolated eigenvalue
for all $\omega \in (-1,1)$ and this eigenvalue is negative for $\omega \in (-1,0)$,
zero for $\omega = 0$, and positive for $\omega \in (0,1)$.
\end{proof}

\begin{lemma}
There is $\omega_0 \in (0,1]$ such that for any fixed $\omega \in (-\omega_0,\omega_0)$,
operator $L_+$ has exactly two eigenvalues below the continuous spectrum. Besides the zero eigenvalue
associated with the eigenvector in (\ref{subspace-1}), $L_+$ also has a negative eigenvalue for
$\omega \in (0,\omega_0)$ and a positive eigenvalue for $\omega \in (-\omega_0,0)$, which is
associated with the eigenvector in (\ref{subspace-2}) for $\omega = 0$.
\label{lemma-eigenvalues-2}
\end{lemma}

\begin{proof}
Because the double zero eigenvalue of $L_+$ at $\omega = 0$ is isolated from the continuous spectrum
located for $[1,\infty)$, the assertion of the lemma will follow by the Kato's perturbation theory
\cite{Kato} if we can show that the zero eigenvalue is the lowest eigenvalue of $L_+$ at $\omega = 0$
and the end-point of the continuous spectrum does not admit a resonance.

To develop the perturbation theory, we consider the eigenvalue
problem $L_+ {\bf u} = \mu {\bf u}$, where ${\bf u} = (u,\bar{u})$
is an eigenvector and $\mu$ is the spectral parameter. Using the transformation
$$
u(x) = \varphi(x) e^{- 3 i \int_0^x |U_{\omega}(x')|^2 dx'}
$$
where $\varphi$ is a new eigenfunction, we obtain an equivalent spectral problem:
$$
\left[ \begin{array}{cc} -\partial_x^2 + 1-\omega^2 -6 \omega |U_{\omega}|^2 - 3 |U_{\omega}|^4 & -6 \omega W \\
-6 \omega \bar{W} & -\partial_x^2 + 1-\omega^2 -6 \omega |U_{\omega}|^2 - 3 |U_{\omega}|^4 \end{array} \right]
\left[ \begin{array}{c} \varphi \\ \bar{\varphi} \end{array} \right] = \mu \left[ \begin{array}{c} \varphi \\ \bar{\varphi} \end{array} \right],
$$
where
\begin{eqnarray*}
W & = & U_{\omega}^2 e^{6 i \int_0^x  |U_{\omega}(x')|^2 dx'} \\
& = & (1-\omega^2)
\frac{\left( 1 + \omega \cosh\left(2 \sqrt{1 - \omega^2} x\right) +
i \sqrt{1-\omega^2} \sinh\left(2 \sqrt{1 - \omega^2} x\right)\right)^2}{
\left( \omega + \cosh\left( 2 \sqrt{1 - \omega^2} x\right) \right)^3}.
\end{eqnarray*}
Setting now $z := \sqrt{1-\omega^2} x$ and $\mu := (1-\omega^2) \lambda$, we rewrite
the spectral problem in the form
\begin{equation}
\label{third}
\left[ \begin{array}{cc} -\partial_z^2 + 1 + V_1(z) & V_2(z) \\
\bar{V}_2(z) & -\partial_z^2 + 1 + V_1(z) \end{array} \right] \left[ \begin{array}{c} \varphi \\ \bar{\varphi} \end{array} \right]
= \lambda \left[ \begin{array}{c} \varphi \\ \bar{\varphi} \end{array} \right],
\end{equation}
where
$$
V_1(z) := -\frac{3(1-\omega^2)}{(\omega + \cosh(2z))^2}-\frac{6 \omega}{\omega + \cosh(2z)}
$$
and
$$
V_2(z) := - 6 \omega \frac{\left(1 + \omega \cosh(2z) +
i \sqrt{1-\omega^2} \sinh(2 z)\right)^2}{\left( \omega + \cosh(2z) \right)^3}.
$$
The eigenvector (\ref{subspace-1}) in the kernel of $L_+$ yields the eigenvector
$(\varphi_0,\bar{\varphi}_0)$ with
$$
\varphi_0(z) = \frac{\omega \sinh(2 z) + i \sqrt{1-\omega^2} \cosh(2z)}{(\omega + \cosh(2z))^{3/2}},
$$
which exists in the spectral problem (\ref{third}) with $\lambda = 0$
for all $\omega \in (-1,1)$. Now, for $\omega = 0$, $\lambda = 0$ is a double
zero eigenvalue of the spectral problem (\ref{third}). The other eigenvector
is $(\varphi_0,-\bar{\varphi}_0)$ and it corresponds to the eigenvector in (\ref{subspace-2}).
The end-point $\lambda = 1$ of the continuous spectrum of the spectral problem (\ref{third})
does not admit a resonance
for $\omega = 0$, which follows from the comparison results in Lemma \ref{lemma-eigenvalues-1}.
No other eigenvalues exist for $\omega = 0$.

To study the splitting of the double zero eigenvalue
if $\omega \neq 0$, we compute the quadratic form of the operator on the left-hand side
of the spectral problem (\ref{third}) at the vector $(\varphi_0,-\bar{\varphi}_0)$ to obtain
$$
-2 \int_{\mathbb{R}} \left( V_2 + \bar{V}_2 \right) |\varphi_0|^2 dz = - 12 \omega \int_{\R}
\frac{-3 + 2 \omega^2 + \cosh(4z)}{(\omega + \cosh(2z))^4} dz.
$$
Since the integral is positive for $\omega = 0$, Kato's perturbation theory (Theorem D in Appendix B)
implies that the zero eigenvalue of the spectral problem (\ref{third}) becomes negative for $\omega < 0$
and positive for $\omega > 0$ with sufficiently small $|\omega|$.
\end{proof}

\begin{conjecture}
The spectral problem (\ref{third}) has exactly two isolated eigenvalues
and no end-point resonances for all $\omega \in (-1,1)$.
The non-zero eigenvalue is positive for all $\omega \in (-1,0)$
and negative for all $\omega \in (0,1)$.
\end{conjecture}

\begin{lemma}
There is $\omega_0 \in (0,1]$ such that
for any $\omega \in (-\omega_0,\omega_0)$,
operators $L_{\pm}$ have no negative eigenvalues and a simple zero eigenvalue
in the constrained spaces $X_{\pm}$ defined by
\begin{eqnarray}
\label{constraints-1}
X_+ & := &  \left\{ u \in L^2(\R) : \quad \int_{\R} \left( \bar{U}_{\omega} u + U_{\omega} \bar{u} \right) dx = 0 \right\}, \\
\label{constraints-2}
X_- & := &  \left\{ u \in L^2(\R) : \quad \int_{\R} \left( \bar{U}_{\omega}' u - U_{\omega}' \bar{u} \right) dx = 0 \right\}.
\end{eqnarray}
For operator $L_-$, the result extends to all $\omega \in (-1,1)$.
\label{lemma-constraints}
\end{lemma}

\begin{proof}
We use Theorem E in Appendix B and compute the value of $\sigma$ in this theorem
explicitly both for operators $L_+$ and $L_-$.

For operator $L_+$, the constraint (\ref{constraints-1}) yields the vector
${\bf s} = (U_{\omega},\bar{U}_{\omega})$. By taking derivative of the second-order
differential equation (\ref{ODE-second-order}) with respect to $\Omega$, we obtain
\begin{equation}
\label{generalized-1}
L_+ \left[ \begin{array}{c} \partial_{\Omega} U_{\omega} \\ \partial_{\Omega} \bar{U}_{\omega} \end{array} \right] =
-\left[ \begin{array}{c} U_{\omega} \\ \bar{U}_{\omega} \end{array} \right],
\end{equation}
hence
\begin{eqnarray*}
\sigma = -\int_{\R} \left( \bar{U}_{\omega} \frac{\partial U_{\omega}}{\partial \Omega} +
U_{\omega} \frac{\partial \bar{U}_{\omega}}{\partial \Omega} \right) dx
= \frac{1}{2\omega} \frac{d}{d\omega} \int_{\R} |U_{\omega}|^2 dx = -\frac{1}{2\omega \sqrt{1-\omega^2}},
\end{eqnarray*}
where we have used the exact expressions $\Omega = 1 - \omega^2$ and $\| U_{\omega} \|^2 = \arccos(\omega)$.
We verify that $\sigma > 0$ for $\omega \in (-1,0)$ and $\sigma < 0$ for $\omega \in (0,1)$.
By Lemma \ref{lemma-eigenvalues-2}, $L_+$ has no negative eigenvalues for $\omega \in (-\omega_0,0)$ and
has one negative eigenvalue for $\omega \in (0,\omega_0)$, whereas the eigenvector
$(U_{\omega}',\bar{U}_{\omega}')$ for zero eigenvalue of $L_+$
is orthogonal to the vector ${\bf s} = (U_{\omega},\bar{U}_{\omega})$. Conditions
of Theorem E in Appendix B are satisfied and $L_+$ has no negative eigenvalues and a simple zero eigenvalue
in the constrained space $X_+$ for all $\omega \in (-\omega_0,\omega_0)$. Note that the result
holds also for $\omega = 0$, since the eigenvector $(U_0',-\bar{U}_0')$ in (\ref{subspace-2}) does not
belong to the constrained space $X_+$ because
\begin{equation}
\label{computation}
\omega = 0 : \quad \int_{\mathbb{R}} \left( \bar{U}_0 U_0' - U_0 \bar{U}_0' \right) dx =
-i \int_{\mathbb{R}} \left( 4 |U_0|^4 - U_0^2 - \bar{U}_0^2 \right) dx = -2i \neq 0.
\end{equation}

For operator $L_-$, the constraint (\ref{constraints-2}) yields the vector
${\bf s} = (U_{\omega}',-\bar{U}_{\omega}')$. By using the differential equations (\ref{ODE}) and
(\ref{ODE-second-order}), we obtain
\begin{equation}
\label{generalized-2}
L_- \left( -\frac{1}{2} x \left[ \begin{array}{c} U_{\omega} \\ -\bar{U}_{\omega} \end{array} \right]
+ \frac{1}{4 i \omega} \left[ \begin{array}{c} U_{\omega} \\ \bar{U}_{\omega} \end{array} \right] \right) =
\left[ \begin{array}{c} U_{\omega}' \\ -\bar{U}_{\omega}' \end{array} \right],
\end{equation}
hence
\begin{eqnarray*}
\sigma & = & \int_{\R} \left( \frac{1}{2} |U_{\omega}|^2 - \frac{1}{4 i \omega}
\left( \bar{U}_{\omega} U_{\omega}' - U_{\omega} \bar{U}_{\omega}' \right) \right) dx \\
& = & \frac{1}{4\omega} \int_{\R} \left( 4 |U_{\omega}|^4 - U_{\omega}^2 - \bar{U}_{\omega}^2 + 4 \omega |U_{\omega}|^2 \right) dx \\
& = & \frac{1-\omega^2}{2\omega} \int_{\R} \frac{1 + \omega \cosh(2 \sqrt{1-\omega^2} x)}{(\omega + \cosh(2 \sqrt{1-\omega^2} x))^2} dx \\
& = & \frac{\sqrt{1 - \omega^2}}{2 \omega}.
\end{eqnarray*}
We verify that $\sigma < 0$ for $\omega \in (-1,0)$ and $\sigma >  0$ for $\omega \in (0,1)$.
By Lemma \ref{lemma-eigenvalues-1}, $L_-$ has one negative eigenvalue for $\omega \in (-\omega_0,0)$ and
no negative eigenvalues for $\omega \in (0,1)$, whereas the eigenvector $(U_{\omega},-\bar{U}_{\omega})$
for zero eigenvalue of $L_-$
is orthogonal to the vector ${\bf s} = (U_{\omega}',-\bar{U}_{\omega}')$. Conditions
of Theorem E in Appendix B are satisfied and $L_-$ has no negative eigenvalues and a simple zero eigenvalue
in the constrained space $X_-$ for all $\omega \in (-\omega_0,\omega_0)$. Again, the result
holds also for $\omega = 0$, since the eigenvector $(U_0,\bar{U}_0)$ of $L_-$ in (\ref{subspace-2}) does not
belong to the constrained space $X_-$ because of the same computation (\ref{computation}).
\end{proof}

\begin{remark}
Solutions of the inhomogeneous equations (\ref{generalized-1}) and (\ref{generalized-2}) define
so-called generalized eigenvectors associated with the zero eigenvalue of the spectral stability problem
associated with MTM solitons of the MTM system (\ref{MTM}).
These solutions are related to translation of the soliton orbit with respect to parameters $\omega$ and $c$.
Indeed, from the Lorentz transformation (\ref{Lorentz}), we realize that the solution (\ref{generalized-2})
is related to the derivative of the MTM soliton with respect to parameter $c$ at $c = 0$.
\end{remark}

The proof of Theorem \ref{theorem-minimum} follows from Lemma \ref{lemma-constraints} and the fact that
the eigenvectors (\ref{subspace-1}) for the zero eigenvalues of $L_+$ and $L_-$ are removed by additing
additional constraints
\begin{eqnarray}
\label{constraints-3}
\tilde{X}_+ & := &  \left\{ u \in L^2(\R) : \quad \int_{\R} \left( \bar{U}_{\omega}' u + U_{\omega}' \bar{u} \right) dx = 0 \right\}, \\
\label{constraints-4}
\tilde{X}_- & := &  \left\{ u \in L^2(\R) : \quad \int_{\R} \left( \bar{U}_{\omega} u - U_{\omega} \bar{u} \right) dx = 0 \right\},
\end{eqnarray}
which are associated with the eigenvectors (\ref{subspace-1}).

Note that the constraints in $X_+$ and $\tilde{X}_-$ give real and imaginary parts of the
complex-valued constraint (\ref{constraint-1}),  whereas
the constraints in $\tilde{X}_+$ and $X_-$ give real and imaginary parts of the
complex-valued constraint (\ref{constraint-2}). Therefore, the
Hessian operator $L$ in (\ref{Hessian}) is strictly positive under the complex-valued
constraints (\ref{constraint-1}) and (\ref{constraint-2}) for $\omega \in (-\omega_0,\omega_0)$
and the proof of Theorem \ref{theorem-minimum} is complete.

The proof of Theorem \ref{theorem-main} is based on the conservation of functionals $R$, $Q$, and $P$
for a solution of the massive Thirring model (\ref{MTM}) in $H^1(\R,\C^2)$ and the standard orbital stability
arguments (Theorem F in Appendix B).

\appendix
\section{Conserved quantities by the inverse scattering method}

The MTM system (\ref{MTM}) is a compatibility condition of the Lax system
\begin{equation}
\label{Lax}
\frac{\partial}{\partial x} \vec{\phi} = L \vec{\phi}, \quad \frac{\partial}{\partial t} \vec{\phi} = A \vec{\phi},
\end{equation}
where $\vec{\phi}(x,t) : \R \times \R \to \C^2$ and $L$ is given by \cite{KL1,KM}:
\begin{equation}
L=\frac{i}{2}(|v|^2-|u|^2)\sigma_3-\frac{i\lambda}{\sqrt{2}}\left(\begin{matrix} 0 & \overline{v} \\ v & 0 \end{matrix}\right) - \frac{i}{\sqrt{2}\lambda} \left(\begin{matrix} 0 & \overline{u} \\ u & 0 \end{matrix}\right) +\frac{i}{4}\left(\frac{1}{\lambda^2}-\lambda^2\right)\sigma_3.
\end{equation}

Let us consider a Jost function $\vec{\phi}(x;\lambda)$, which satisfies the boundary condition
\begin{equation} \label{phicon}
\lim_{x \to -\infty} e^{-ik(\lambda)x} \vec{\phi}(x;\lambda) = \left[\begin{matrix} 1 \\ 0 \end{matrix} \right],
\end{equation}
where $k(\lambda) := \frac{1}{4}(\lambda^{-2} - \lambda^2) \in \mathbb{R}$ if $\lambda^2 \in \mathbb{R}$.
This Jost function satisfies the scattering
relation as $x \to +\infty$,
\begin{equation}
\vec{\phi}(x;\lambda) \sim a(\lambda) e^{ik(\lambda)x} \left[\begin{matrix} 1 \\ 0 \end{matrix} \right]
+ b(\lambda) e^{-ik(\lambda)x} \left[\begin{matrix} 0 \\ 1 \end{matrix} \right],
\end{equation}
where $a(\lambda)$ and $b(\lambda)$ are spectral coefficients for $\lambda^2 \in \mathbb{R}$.

Setting
\begin{equation} \label{solution1}
\vec{\phi}(x;\lambda) =\left[\begin{matrix} 1 \\ \nu(x;\lambda) \end{matrix} \right]
\exp\left(i k(\lambda)x+\int_{-\infty}^{x}\chi(x';\lambda)dx' \right)
\end{equation}
with two functions $\chi(x;\lambda)$ and $\nu(x;\lambda)$ satisfying the boundary conditions
$$
\lim_{x \rightarrow -\infty} \chi(x;\lambda) = 0 \quad \mbox{\rm and}
\quad \lim_{x\rightarrow -\infty} \nu(x; \lambda)=0
$$
and taking a limit $x\rightarrow \infty$ in \eqref{solution1}, we obtain
\begin{equation}
a(\lambda) =  \exp\left(\int_{-\infty}^{\infty}\chi(x;\lambda)dx \right) \quad \Rightarrow \quad
\log a(\lambda)=\int_{-\infty}^{\infty}\chi(x;\lambda)dx.
\end{equation}

The scattering coefficient $a(\lambda)$ does not depend on time $t$,
hence expansion of $\int_{-\infty}^{\infty}\chi(x;\lambda)dx$ in powers of $\lambda$
yields conserved quantities with respect to $t$  \cite{KM}.
Substituting equation \eqref{phicon} into the $x$-derivative part of the Lax system (\ref{Lax}),
we obtain
\begin{equation} \label{chieq}
\chi = \frac{i}{2} (|v|^2 - |u|^2) -
\frac{i}{\sqrt{2}}\left(\lambda \overline{v}+\frac{1}{\lambda}\overline{u}\right) \nu,
\end{equation}
where $\nu$ satisfies a Ricatti equation
\begin{equation} \label{riccatiA}
\nu_x + i\left(2 k(\lambda) + |v|^2 - |u|^2\right) \nu
-\frac{i}{\sqrt{2}}\left(\lambda \overline{v}+\frac{1}{\lambda}\overline{u}\right) \nu^2
+\frac{i}{\sqrt{2}}\left(\lambda v+\frac{1}{\lambda}u\right)=0.
\end{equation}

To generate two hierarchies of conserved quantities, we consider the formal asymptotic expansion of
$\chi(x;\lambda)$ in powers and inverse powers of $\lambda$:
\begin{equation}
\chi(x;\lambda)=\sum_{n=0}^{\infty}\lambda^n\chi_n(x), \quad \nu(x;\lambda)=\sum_{n=1}^{\infty}\lambda^n \nu_n(x)
\end{equation}
and
\begin{equation}
\chi(x;\lambda)=\sum_{n=0}^{\infty}\frac{1}{\lambda^n}\widetilde{\chi}_{n}(x), \quad
\nu(x;\lambda)=\sum_{n=1}^{\infty}\frac{1}{\lambda^n}\widetilde{\nu}_{n}(x).
\end{equation}
We set
\begin{equation}
I_n := \int_{-\infty}^{\infty}\chi_n(x)dx, \quad I_{-n} := \int_{-\infty}^{\infty}\widetilde{\chi}_n(x)dx.
\end{equation}
Substitution of the asymptotic expansion of $\nu$ into the Riccati equation \eqref{riccatiA} allows one
to determine each $\nu_n$ and $\widetilde{\nu}_{n}$ from which equation \eqref{chieq}
is used to determine $\chi_n$ and $\widetilde{\chi}_n$. Let us explicitly write out first conserved quantities
$$
I_0=\int_{\mathbb{R}}(|u|^2+|v|^2)dx,
$$
$$
I_2=\int_{\mathbb{R}}(-2u_x\overline{u}+i\overline{v}u+i\overline{u}v-2i|u|^2|v|^2)dx,
$$
$$
I_{-2}=\int_{\mathbb{R}}(-2v_x\overline{v}-i\overline{v}u-i\overline{u}v+2i|u|^2|v|^2)dx,
$$
\begin{eqnarray*}
I_4 & = & \int_{\mathbb{R}}[-4i\overline{u}u_{xx}-2(u_x\overline{v}+\overline{u}v_x)+
4\overline{u}(u|v|^2)_x+4u_x\overline{u}(|u|^2+|v|^2)+i(|u|^2+|v|^2) \\
& \phantom{t} & \phantom{text} -2iu\overline{v}(|u|^2+|v|^2) - 2iv\overline{u}(|u|^2+|v|^2)+4i|u|^2|v|^2(|u|^2+|v|^2)]dx,
\end{eqnarray*}
and
\begin{eqnarray*}
I_{-4} & = & \int_{\mathbb{R}} [4i\overline{v}v_{xx}-2(u_x\overline{v}+\overline{u}v_x)+4\overline{v}(v|u|^2)_x+
4v_x\overline{v}(|u|^2+|v|^2)-i(|u|^2+|v|^2)\\
& \phantom{t} & \phantom{text} +2iu\overline{v}(|u|^2+|v|^2) + 2iv\overline{u}(|u|^2+|v|^2)-4i|u|^2|v|^2(|u|^2+|v|^2)]dx.
\end{eqnarray*}

We note that $I_0$ corresponds to charge $Q$ in (\ref{quantity-Q}). After integration by parts,
$I_2+I_{-2}$ corresponds to momentum $P$ in (\ref{quantity-P}) and $I_2-I_{-2}$ corresponds
to Hamiltonian $H$ in (\ref{quantity-H}). The higher-order Hamiltonian $R$ in (\ref{quantity-R})
is obtained from $I_4-I_{-4}$ after integration by parts and dropping the conserved quantity $Q$
from the definition of $R$.

Using Wolfram's MATHEMATICA, we also obtain the balance equation for $R$:
\begin{equation}
\label{balance-R}
\frac{\partial \rho}{\partial t} + \frac{\partial j}{\partial x} = 0,
\end{equation}
where
\begin{eqnarray*}
\rho & = & |u_{x}|^2+|v_x|^2 - \frac{i}{2}(u_x\overline{u}-\overline{u}_xu)(|u|^2+2|v|^2) +
\frac{i}{2}(v_x\overline{v}-\overline{v}_xv)(2|u|^2+|v|^2) \\
& \phantom{t} & -(u\overline{v}+\overline{u}v)(|u|^2+|v|^2)+2|u|^2|v|^2(|u|^2+|v|^2)
\end{eqnarray*}
and
\begin{eqnarray*}
j & = & |u_{x}|^2 - |v_x|^2 - \frac{i}{2}(u_x\overline{u}-\overline{u}_xu)(|u|^2+2|v|^2) -
\frac{i}{2}(v_x\overline{v}-\overline{v}_xv)(2|u|^2+|v|^2)\\
& \phantom{t} & - \frac{1}{2} (u\overline{v}+\overline{u}v)(|u|^2-|v|^2).
\end{eqnarray*}

\section{Auxiliary results used in this work}

We are using the following technical results in the main part of this article.
In the next four results, we consider a linear Schr\"{o}dinger operator
$L : H^2(\mathbb{R}) \to L^2(\mathbb{R})$ given by
$$
L:= -\partial_x^2 + c + V(x),
$$
where $V \in L^1(\mathbb{R}) \cap L^{\infty}(\mathbb{R})$ and $c > 0$ is fixed.

\vspace{0.25cm}

\renewcommand{\theorem}{{\bf Theorem A }}
\begin{theorem}{\cite[Lemma 4.2]{Pel-book}}{\bf .}
There exists a unique solution of $L u_0 = \lambda_0 u_0$ for any $\lambda_0 \leq c$
such that $u_0 \in H^2_{\rm loc}(\mathbb{R})$ and $\lim_{x\to -\infty} e^{-\sqrt{c-\lambda_0} x} u_0(x) = 1$.
If $u_0$ has $n(\lambda_0)$ zeros on $\mathbb{R}$, then there exists exactly $n(\lambda_0)$ eigenvalues of $L$
for any $\lambda < \lambda_0$.
\end{theorem}

\vspace{0.25cm}

\renewcommand{\theorem}{{\bf Theorem B }}
\begin{theorem}{\cite[Theorem B.10]{Pel-book}}{\bf .}
Let $u(x;V)$ be a solution of $L u = c u$
such that $\lim_{x\to -\infty} u(x;V) = 1$. Assume that $V_1(x) > V_2(x)$ for all $x \in \mathbb{R}$
and $u(x;V_2)$ has one zero on $\mathbb{R}$. Then, $u(x;V_1)$ has at most one zero on $\mathbb{R}$.
\end{theorem}

\vspace{0.25cm}

\renewcommand{\theorem}{{\bf Theorem C }}
\begin{theorem}{\cite[Section I.6.10]{Kato}}{\bf .}
There exists the smallest eigenvalue $\lambda_0 < c$ of $L$ if and only if
$$
\lambda_0 = \inf_{u \in H^1(\mathbb{R}) : \| u \|_{L^2} = 1} \int_{\mathbb{R}}
\frac{1}{2} \left[ (\partial_x u)^2 + c u^2 + V(x) u^2 \right] dx < c.
$$
In particular, if $\lambda_0 = 0$ for $V = V_0$ and $\Delta V > 0$, then
$\lambda_0 \gtrless 0$ for $V = V_0 \pm \Delta V$.
\end{theorem}

\vspace{0.25cm}

\renewcommand{\theorem}{{\bf Theorem D }}
\begin{theorem}{\cite[Section VII.4.6]{Kato}}{\bf .}
Let $\lambda_0 < c$ be an isolated eigenvalue of $L$ with the eigenfunction $u_0 \in H^2(\mathbb{R})$.
Then, the perturbed operator $\tilde{L} := L + \Delta V$ with $\Delta V \in L^{\infty}(\mathbb{R})$ has
a perturbed eigenvalue $\tilde{\lambda}_0$ near $\lambda_0$ and the sign of $\tilde{\lambda_0} - \lambda_0$
coincides with the sign of the quadratic form $\langle \Delta V u_0, u_0 \rangle_{L^2}$.
\end{theorem}

\vspace{0.25cm}

\renewcommand{\theorem}{{\bf Theorem E }}
\begin{theorem}{\cite[Theorem 4.1]{Pel-book}}{\bf .}
Assume that $H$ is a Hilbert space equipped with the inner product $\langle \cdot, \cdot \rangle$.
Fix a vector ${\bf s} \in H$. Assume that $L$ is a self-adjoint operator on $H$ such that the
number of negative eigenvalues of $L$ is $n(L)$, the eigenvectors of $L$
for the zero eigenvalue are orthogonal to ${\bf s}$, and the rest of the spectrum of $L$
is bounded away from zero. Then, the number of negative eigenvalues of $L$ in
the constrained space
$$
H_c := \{ {\bf u} \in H : \quad \langle {\bf s},{\bf u} \rangle = 0 \}
$$
is defined by the sign of
$$
\sigma := \langle L^{-1} {\bf s}, {\bf s} \rangle.
$$
If $\sigma > 0$, then the number of negative eigenvalues of $L$ under the constraint is $n(L)$,
whereas if $\sigma < 0$, then this number is $n(L)-1$.
\end{theorem}

\vspace{0.25cm}

\renewcommand{\theorem}{{\bf Theorem F }}
\begin{theorem}{\cite[Theorem 4.15]{Pel-book}}{\bf .}
Let $X=H^1(\mathbb{R},\mathbb{C}^4)$ be the energy space for the solution
$\vec{\psi} := (u,v)$ of the massive Thirring model (\ref{MTM}).
Let $\vec{\phi}_{\omega} :=(U_{\omega},\overline{U}_{\omega})$
be a a local non-degenerate minimizer of $R$ in $X$ under the constraints of
fixed $Q$ and $P$ for some $\omega \in (-1,1)$. Then, $\vec{\phi}_{\omega}$ is orbitally
stable in $X$ with respect to the time evolution of the MTM system (\ref{MTM}).
In other words, for any $\epsilon > 0$, there is $\delta > 0$ such that
if the initial datum satisfies
\begin{equation*}
\inf_{\alpha,\beta \in \mathbb{R}} \|\vec{\psi} |_{t = 0} - e^{i \alpha} \vec{\phi}_{\omega}(\cdot+\beta)\|_{X} < \delta,
\end{equation*}
then for all $t > 0$, we have
\begin{equation*}
\inf_{\alpha,\beta \in \mathbb{R}} \|\vec{\psi} - e^{i \alpha} \vec{\phi}_{\omega}(\cdot+\beta)\|_{X} < \epsilon.
\end{equation*}
\end{theorem}

\end{document}